\documentclass[11pt]{amsart}
\usepackage[colorlinks,citecolor=red, dvipdfm]{hyperref}

\newtheorem{theorem}{Theorem}[section]
\newtheorem{lemma}[theorem]{Lemma}

\theoremstyle{definition}
\newtheorem{definition}[theorem]{Definition}

\newtheorem{example}[theorem]{Example}

\newtheorem{proposition}[theorem]{Proposition}
\newtheorem{corollary}[theorem]{Corollary}
\newtheorem{remark}[theorem]{Remark}

\theoremstyle{remark}

\newcommand{\bs}{\begin{split}}
\newcommand{\es}{\begin{split}}

\newcommand{\be}{\begin{equation}}
\newcommand{\ee}{\end{equation}}

\numberwithin{equation}{section}

\begin{document}

\title[Eigenvalues, residue formula and integral invariants]
{Eigenvalues of vector fields, Bott's residue formula and integral
invariants}

\author{Ping Li}
\address{Department of Mathematics, Tongji University, Shanghai 200092, China}
%\address{Department of Mathematics, Faculty of Science and Engineering, Waseda University, Shinjuku, Tokyo 169-8555, Japan}
%\curraddr{Department of Mathematics and Information Sciences, Tokyo
%Metropolitan University, Tokyo 192-0397, Japan}
%\curraddr{Department of Mathematics and Information Sciences, Tokyo
%Metropolitan University, Tokyo 192-0397, Japan}
\email{pingli@tongji.edu.cn,\qquad pinglimath@gmail.com}
\thanks{The author was partially supported by the National
Natural Science Foundation of China (Grant No. 11471247) and the
Fundamental Research Funds for the Central Universities.}

%    General info

%\date{January 1, 2001 and, in revised form, June 22, 2001.}


\begin{abstract}
Given a compatible vector field on a compact connected
almost-complex manifold, we show in this article that the
multiplicities of eigenvalues among the zero point set of this
vector field have intimate relations. We highlight a special case of
our result and reinterpret it as a vanishing-type result in the
framework of the celebrated Atiyah-Bott-Singer localization formula.
This new point of view, via the Chern-Weil theory and a strengthened
version of Bott's residue formula observed by Futaki and Morita, can
lead to an obstruction to Killing real holomorphic vector fields on
compact Hermitian manifolds in terms of a curvature integral.
\end{abstract}

\maketitle

\tableofcontents
\section{Introduction}\label{section1}
In an earlier article \cite{Li3}, the author showed that, if a
compact connected almost-complex manifold admits a compatible circle
action with nonempty fixed points, the weights among the fixed point
set of this action have intimate relations \cite[Theorem 1.1]{Li3}.
The main idea of the proof in \cite{Li3} is refined from a beautiful
observation of Lusztig in \cite{Lu}, which is the invariance of the
equivariant Hirzebruch $\chi_y$-genus under compact connected Lie
group actions. Recently, the author also notices that Lusztig's this
consideration is closely related to Futaki and Morita's work on the
reinterpretation of the Futaki integral invariant on Fano manifolds
(\cite{FM1}, \cite{FM2}, \cite[Chapter 5]{Fu1}). In fact some
considerations of Futaki and Morita could be improved by using the
observation in \cite{Lu} and some other results, which has been
clarified by the author in \cite{Li2}.

The main purpose of the present article is twofold. On the one hand,
we shall show that the idea developed in \cite{Li3} can be carried
over to the case of vector fields on compact almost-complex
manifolds to yield a similar result (Theorem \ref{mainresult1}). On
the other hand, we highlight a special case of this result
(Corollary \ref{corollary1}) by rephrasing it as a vanishing-type
result in the framework of the Atiyah-Bott-Singer localization
formula. This new point of view, through the Chern-Weil theory and a
strengthened version of Bott's residue formula, can lead to a
vanishing result of curvature integral when the underlying manifolds
are complex and the vector fields are Killing (Theorem
\ref{mainresult2}).

Here it is worth pointing out that the idea of this application has
a new feature, which is converse to the usual philosophy of
localization methods involving vector fields. The latter is to
localize the investigation of a global property of the manifolds to
the consideration of local information around the zero points of
vector fields. While our method is to show a property related to
some local information around the zero points of vector fields and
piece it up into a global one.

The rest of this article is arranged as follows. In Section
\ref{section1.5} we will state our main results in this article,
Theorems \ref{mainresult1} and \ref{mainresult2} and Corollaries
\ref{corollary1} and \ref{corollary2}. We recall the
Atiyah-Bott-Singer localization formula in Section \ref{section2}
and reinterpret our Corollary \ref{corollary1} as a vanishing-type
result in this framework. In Section \ref{section3} we firstly
introduce a strengthened version of Bott's residue formula, which
includes the invariant polynomials whose degrees are larger than the
dimension of the underlying manifold and has been observed by Futaki
and Morita to reinterpret the famous Futaki invariant. Then
combining this version of Bott's residue formula with our
reinterpretation of Corollary \ref{corollary1} leads to the proof of
Theorem \ref{mainresult2}. Although the main idea of the proof of
Theorem \ref{mainresult1} is similar to that as in \cite{Li3}, for
completeness and the reader's convenience, we still in Section
\ref{section4} present its proof.

\section{Statement of main results}\label{section1.5}
Suppose $(M^{n},J)$ is a compact connected almost-complex manifold
with complex dimension $n$ and a fixed almost-complex structure $J$.
A smooth vector field $A$ on $(M^{n},J)$ is called \emph{compatible}
if it preserves the almost-complex structure $J$ and the one
parameter group of $A$, $\textrm{exp}(tA)$, lies in a \emph{compact}
group. The latter condition is equivalent to the condition that the
vector field $A$ preserves an almost-Hermitian metric on
$(M^{n},J)$, i.e., $A$ is \emph{Killing} with respect to this
almost-Hermitian metric.

Given this $(M^{n},J)$ and a compatible vector field $A$ whose zero
point set is \emph{nonempty}, we choose an almost-Hermitian metric
$g$ such that $A$ is Killing with respect to this $g$.  Let
$\textrm{zero}(A)$ denote the zero point set of $A$. As is
well-known in the compact transformation group theory (\cite{Ko}),
$\textrm{zero}(A)$ consists of finitely many connected components
and each one is a compact almost-Hermitian submanifold in $M$.
Moreover, the normal bundle of each connected component in
$\textrm{zero}(A)$ can be splitted into a sum of complex line
bundles with respect to the skew-Hermitian transformation induced by
$A$. Let $Z$ be any such a connected component with complex
dimension $r$. Here $r$ of course depends on the choice of $Z$ in
$\textrm{zero}(A)$. Then the normal bundle of $Z$ in $M$, denoted
$\nu(Z)$, can be decomposed into a sum of $n-r$ complex line bundles
$$\nu(Z)=\bigoplus_{i=1}^{n-r}L(Z,\lambda_i),\qquad \lambda_i\in\mathbb{R}-\{0\},$$

such that the eigenvalue of the skew-Hermitian transformation
induced by $A$ on the line bundle $L(Z,\lambda_i)$ is
$\sqrt{-1}\lambda_i$. Or equivalently, the eigenvalue of the action
induced by the one-parameter group $\textrm{exp}(tA)$ on
$L(Z,\lambda_i)$ is $\textrm{exp}(\sqrt{-1}\lambda_it)$. Note that
these nonzero real numbers $\lambda_1,\ldots,\lambda_{n-r}$ are
counted with multiplicities and thus not necessarily mutually
distinct. Of course they depend on the choice of $Z$ in
$\textrm{zero}(A)$. Note also that these $\lambda_i$ are actually
independent of the almost-Hermitian metric $g$ we choose and
completely determined by the vector field $A$.

From now on we use $``e(\cdot)"$ to denote the Euler characteristic
of a manifold.

\begin{definition}\label{definition}
Let us attach a set $\textrm{S}(A)$ to the vector field $A$ as
follows.
\begin{eqnarray}
\textrm{S}(Z): = \left\{ \begin{array}{ll}
\coprod_{e(Z)~\textrm{copies}}\{\lambda_1,\ldots,\lambda_{n-r}\}, &
\textrm{if
$e(Z)>0$},\\
\emptyset, & \textrm{if $e(Z)=0$},\\
\coprod_{-e(Z)~\textrm{copies}}\{-\lambda_1,\ldots,-\lambda_{n-r}\},
& \textrm{if $e(Z)<0$}.
\end{array} \right.
\nonumber\end{eqnarray}

and
$$\textrm{S}(A):=\coprod_{Z}\textrm{S}(Z),$$

where the sum is over all the connected components in
$\textrm{zero}(A)$. Here the symbol $``\coprod"$ means
\emph{disjoint union}. That means, although we write
$\{\lambda_1,\ldots,\lambda_{n-r}\}$ as a set, repeated elements in
it may \emph{not} be discarded.

\end{definition}

Now we can state our main result in this section, which shows that
the eigenvalues on the connected components of $\textrm{zero}(A)$
whose Euler numbers are nonzero have intimate relations.

\begin{theorem}\label{mainresult1}
{\rm{Let the notation be as above. Then for any
$\lambda\in\textrm{S}(A)$, the multiplicity of $\lambda$ in
$\textrm{S}(A)$ is the same as that of $-\lambda$ in
$\textrm{S}(A)$.}}
\end{theorem}

By our Definition \ref{definition}, Theorem \ref{mainresult1} itself
provides no information on those connected components whose Euler
numbers are zero. However, the following direct corollary, which
will be rephrased as a vanishing-type result in terms of the
Atiyah-Bott-Singer localization formula in the next section and play
an important role in our applications, can include them.

\begin{corollary}\label{corollary1}
With the above notation understood, we have
$$\sum_{Z}\big[e(Z)\cdot\sum_{i=1}^{n-r}\lambda_i\big]=0,$$

the sum being over all the connected components in
$\textrm{zero}(A)$.
\end{corollary}

\begin{remark}
~
\begin{enumerate}
\item
When $\text{zero}(A)$ consists of isolated zero points, various
special cases of Theorem \ref{mainresult1} and Corollary
\ref{corollary1} have been obtained in previous literature by using
different methods (\cite[Theorem 2]{PT}, \cite[Theorem 3.5]{Li},
\cite[Corollary 6.3]{FM1}, \cite[Corollary 5.3.12]{Fu1}).

\item
When $\text{zero}(A)$ is arbitrary and the compatible vector field
$A$ generates a circle action, it has been obtained by the author in
\cite{Li3}.
\end{enumerate}
\end{remark}

The following typical example illustrates Theorem \ref{mainresult1}
very well.

\begin{example}
Let $\mathbb{C}P^n$ be the $n$-dimensional complex projective space
with homogeneous coordinate $[z_0,z_1,\ldots,z_n]$ and
$\lambda_1,\ldots,\lambda_s$ ($s\leq n+1$) be $s$ mutually distinct
real numbers. We arbitrarily choose $s$ nonnegative integers
$n_1,\ldots,n_s$ such that
$$\sum_{i=1}^s(n_i+1)=n+1.$$

 Using these data we can
define a one-parameter group action $\psi_t$ on $\mathbb{C}P^n$ by
$$\psi_t~:~\mathbb{C}P^n\longrightarrow\mathbb{C}P^n,$$
$$[z_0,z_1,\ldots,z_n]\longmapsto$$
$$[e^{\sqrt{-1}\lambda_1t}z_0,\ldots,
e^{\sqrt{-1}\lambda_1t}z_{n_1},
e^{\sqrt{-1}\lambda_2t}z_{n_1+1},\ldots,
e^{\sqrt{-1}\lambda_2t}z_{n_1+n_2+1}, \ldots,
e^{\sqrt{-1}\lambda_st}z_{n-n_s},\ldots,
e^{\sqrt{-1}\lambda_st}z_n],$$

 i.e., each $\lambda_i$ appears
exactly $n_i+1$ times consecutively. Let $A$ be the vector field
generating this $\psi_t$. Thus
$$\textrm{zero}(A)=\textrm{fixed point set of the action $\{\psi_t\}$}
=\coprod_{i=1}^sM_i,$$

where
$$M_i=\{[0,\ldots,0,z_{k_i},\ldots,z_{k_i+n_i},0,\ldots,0]\in
\mathbb{C}P^{n}\}\cong\mathbb{C}P^{n_i},$$ and
$$k_i=\sum_{j=1}^{i-1} (n_{j}+1)~(i\geq 2),\qquad k_1:=0,\qquad
\mathbb{C}P^{0}:=\{\emph{\textrm{pt}}\}.$$

The eigenvalues of the vector field $A$ on the connected component
$M_i$ are
$$\{\sqrt{-1}(\lambda_j-\lambda_i)~\textrm{with multiplicity}~ n_j+1~\big|~j\neq
i\}.$$

Thus in $\textrm{S}(A),$ the multiplicity of each
$\lambda_j-\lambda_i$ ($1\leq i,j\leq s, i\neq j$) is
$(n_i+1)(n_j+1)$ as $e(M_i)=n_i+1$.
\end{example}

To distinguish the symbols $(M^n,J)$ and $A$ for an almost-complex
manifold and its compatible vector field, we use once and for all
$(N,J)=(N^{n},J)$ and $X$ to denote a compact complex manifold of
complex dimension $n$ and a \emph{real holomorphic} vector field on
$N$ respectively. Here by ``real holomorphic'' we mean that the
corresponding vector field $X-\sqrt{-1}JX$ on $T^{1,0}N$, the
$(1,0)$-part of the complexified tangent bundle of $N$, is
holomoprhic.

Before stating our next result, we need to introduce some notation
and symbols.

We choose a Hermitian metric $g$ on $(N,J)$ and let $\nabla$ be the
Hermitian connection on the holomorphic tangent bundle $T^{1,0}N$,
which is also called in some literature the \emph{Chern connection}.
Let $\Gamma(\textrm{End}(T^{1,0}N))$ be the vector space of smooth
sections of the endomorphism bundle of $T^{1,0}N$. Following Bott
(\cite[Lemma 1]{Bo1}), we define an element
$L(X)\in\Gamma(\textrm{End}(T^{1,0}N))$ as follows. Put
$$L(X)(\cdot):=[X,\cdot]-\nabla_X(\cdot):~\Gamma(T^{1,0}N)\rightarrow \Gamma(T^{1,0}N),$$

where $[\cdot,\cdot]$ is the Lie bracket and $\Gamma(T^{1,0}N)$ is
the vector space of smooth sections of $T^{1,0}N$.  The stability of
$\Gamma(T^{1,0}N)$ under the map $L(X)$ has been explained in
\cite[Lemma 1]{Bo1}. Moreover, for any smooth function $f$ and
$Y\in\Gamma(T^{1,0}N)$, it is direct to verify that $L(X)(fY)=
fL(X)(Y)$ by using the derivation property of $\nabla$. Thus $L(X)$
can be viewed as an $\textrm{End}(T^{1,0}N)$-valued function, i.e.,
$L(X)\in\Gamma(\textrm{End}(T^{1,0}N)).$ Let $R$ be the curvature
form of $\nabla$, which is an $\textrm{End}(T^{1,0}N)$-valued
$(1,1)$-form on $N$, i.e., $R\in
\Gamma(\textrm{End}(T^{1,0}N)\otimes
T^{1,0}N\otimes\overline{T^{1,0}N})$. Therefore it makes sense to
discuss the trace $\big(\text{tr}(\cdot)\big)$ and determinant
$\big(\text{det}(\cdot)\big)$ of $L(X)$ and $R$ (see Theorem
\ref{mainresult2} below).

After reinterpreting Corollary \ref{corollary1} as a vanishing-type
result in terms of the Atiyah-Bott-Singer residue formula in Section
\ref{section2} and building a bridge via the Chern-Weil theory in
Section \ref{section3}, we can yield the following result, which
provides an obstruction-type result to the holomorphic Killing
vector fields on compact complex manifolds.

\begin{theorem}\label{mainresult2}
{\rm{Suppose $(N^{n},~J, ~g)$ is an $n$-dimensional compact
Hermitian manifold and $X$ is a Killing and real holomorphic vector
field. Then we have
$$\int_M\textrm{tr}\big(L(X)\big)\cdot c_n(\nabla)+
\int_Mc_1(\nabla)\cdot\textrm{det}
\big(L(X)+\frac{\sqrt{-1}}{2\pi}R\big)=0.$$

 Here $c_i(\nabla)$
is the $i$-th Chern form with respect to the Chern connection
$\nabla$.}}
\end{theorem}

\begin{remark}
Note that $\textrm{det} \big(L(X)+\frac{\sqrt{-1}}{2\pi}R\big)$ is a
differential form of possibly mixed degrees. So in this theorem we
are only concerned with its homogeneous component of
$(n-1,n-1)$-form due to the dimensional reason.
\end{remark}

Theorem \ref{mainresult2} has the following corollary when the
Hermitian metric $g$ is K\"{a}hler.

\begin{corollary}\label{corollary2}
If $(N^{n}, ~J, ~g)$ is an $n$-dimensional compact K\"{a}hler
manifold and $X$ is a Killing vector field, we have
$$\int_M\textrm{Ric}(g)\cdot\textrm{det}
\big(\nabla X+\frac{\sqrt{-1}}{2\pi}R\big)=0,$$

where $\textrm{Ric}(g)$ is the Ricci form of the K\"{a}hler metric
$g$.
\end{corollary}

\begin{proof}
If the Hermitian metric $g$ is K\"{a}hler, then the Chern connection
$\nabla$ coincides with the Levi-Civita connection of $g$, which
implies
$$L(X)(\cdot)=[X,\cdot]- \nabla_X(\cdot)=-\nabla_{(\cdot)}X
=-\nabla X.$$

On the other hand, a well-known fact (\cite[p. 107, Theorem
4.3]{Ko}) tells us that a smooth vector field $X$ on a compact
K\"{a}hler manifold is Killing if and only if it is real holomorphic
and its divergence $\text{div}(X)=0$. By definition we have
$\text{div}(X)=\text{tr}(\nabla X)$ and thus
$\textrm{tr}\big(L(X)\big)=-\textrm{tr}(\nabla X)=0$ if $X$ is
Killing. Also note that in this case the first Chern form
$c_1(\nabla)=\text{Ric}(g)$. Therefore,
 \be\bs
0&=\int_Mc_1(\nabla)\cdot\textrm{det}
\big(L(X)+\frac{\sqrt{-1}}{2\pi}R\big)\\
&=\int_M\text{Ric}(g)\cdot\textrm{det}
\big(-\nabla(X)+\frac{\sqrt{-1}}{2\pi}R\big)\\
&=-\int_M\text{Ric}(g)\cdot\textrm{det}
\big(\nabla(X)+\frac{\sqrt{-1}}{2\pi}R\big).\end{split}\nonumber\ee
\end{proof}

\section{Atiyah-Bott-Singer localization formula and  reinterpretation of Corollary \ref{corollary1}}\label{section2}
As before let $(M^n,J)$ be a compact connected almost-complex
manifold and $A$ its compatible vector field with nonempty zero
point set $\text{zero}(A)$. We keep using the related notation and
symbols introduced in Section \ref{section1.5}.

 In this section, we briefly
recall the Atiyah-Bott-Singer residue formula, which reduces the
calculation of the Chern numbers of $(M^{2n},J)$ to the
consideration of the local information around $\textrm{zero}(A)$.

 Let $x_1,\ldots,x_n$ denote
the formal Chern roots of $M$, i.e., the total Chern class of
$(M,J)$, $c(M,J)$, has the following formal decomposition:
$$c(M,J):=1+\sum_{i=1}^nt^i\cdot c_i(M,J)
=\prod_{i=1}^n(1+tx_i).$$

Similarly, we denote by $\alpha_1,\ldots,\alpha_r$ the formal Chern
roots of the connected component $Z$ in $\text{zero}(A)$ and
$\beta_1,\ldots,\beta_{n-r}$ the Euler classes (or the first Chern
classes) of the complex line bundles
$L(Z,\lambda_1),\ldots,L(Z,\lambda_{n-r})$.

With the above-defined notation and symbols in mind, we have the
following localization formula, which reduces the calculation of the
Chern numbers of $(M^{2n},J)$ to the consideration of the local
information around $\textrm{zero}(A)$ \cite[p. 598]{AS}.

\begin{theorem}[Residue formula, almost-complex case]\label{localization1}
{\rm{Let $\varphi=\varphi(\cdot,\ldots,\cdot)$ be a symmetric
polynomial with $n$ variables and we define
$$f_{\varphi}(A):=\sum_Z\int_Z\frac{\varphi(\alpha_1,\ldots,\alpha_r,\sqrt{-1}\lambda_1+\beta_1,\ldots,
\sqrt{-1}\lambda_{n-r}+\beta_{n-r})}{\prod_{j=1}^{n-r}(\sqrt{-1}\lambda_j+\beta_j)}.$$

If we use $\text{deg $(\varphi)$}$ to denote the degree of the
symmetric polynomial $\varphi$. Then
\begin{eqnarray}
f_{\varphi}(A) = \left\{ \begin{array}{ll}
0, & \text{deg~$(\varphi)<n$}\\
\int_M\varphi(x_1,\ldots,x_n), & \textrm{deg~$(\varphi)=n$.}
\end{array} \right.
\nonumber\end{eqnarray} }}
\end{theorem}
~

\begin{remark}~\label{remark bott}
When $J$ is integrable, i.e., $(M,J)$ is a compact complex manifold,
Theorem \ref{localization1} was established by Bott in \cite[Theorem
1]{Bo1} ($\textrm{zero}(A)$ is isolated) and \cite[Theorem 2]{Bo2}
(general case) by using direct differential-geometric arguments. The
current version was established by Atiyah and Singer in \cite[\S
~8]{AS}, which is a beautiful application of their general Lefschetz
fixed point formula.
\end{remark}

Note that, even if $\textrm{deg}(\varphi)>n$, $f_{\varphi}(A)$ is
still \emph{well-defined}. But Theorem \ref{localization1} says
\emph{nothing} for those $\varphi$ whose degrees are \emph{larger}
than $n$. Our first observation in this section is that our
Corollary \ref{corollary1} is equivalent to a vanishing result of
$f_{\varphi}(A)$ for some $\varphi$ whose degree is $n+1$. To be
more precise, if we use $c_i=c_i(\cdot,\ldots,\cdot)$ ($1\leq i\leq
n$) to denote the $i$-th elementary symmetric polynomial with $n$
variables, Corollary \ref{corollary1} can be rephrased as follows.

\begin{proposition}\label{prop}
$f_{c_1c_n}(A)\equiv 0$ for any compatible vector field $A$ on
$(M^{2n},J)$.
\end{proposition}

\begin{proof}
\be\bs &f_{c_1c_n}(A)\\
=&\sum_Z\int_Z\frac{c_1(\cdots)c_n(\cdots)}
{\prod_{j}^{n-r}(\sqrt{-1}\lambda_j+\beta_j)}\\
=&\sum_Z\int_Z\frac{\big[\sum_{i=1}^r\alpha_i+\sum_{i=1}^{n-r}
(\sqrt{-1}\lambda_i+\beta_i)\big]\cdot
\big[\prod_{i}^{r}\alpha_i\cdot\prod_{i}^{n-r}(\sqrt{-1}\lambda_i+\beta_i)\big]}
{\prod_{j}^{n-r}(\sqrt{-1}\lambda_j+\beta_j)}\\
=&\sum_Z\int_Z\big\{[\sum_{i=1}^r\alpha_i+\sum_{i=1}^{n-r}(\sqrt{-1}\lambda_i+\beta_i)]\cdot
\prod_{i=1}^r\alpha_i\big\}\\
=&\sum_Z[e(Z)\cdot\sum_{j=1}^{n-r}\sqrt{-1}\lambda_j]\\
=&0.\end{split}\nonumber\ee The fourth equality is due to the fact
that $\prod_{i=1}^r\alpha_i$ is nothing but the Euler class of $Z$
and $\text{dim}_{\mathbb{C}}Z=r$.
\end{proof}

\begin{remark}\label{remarkfutakimorita}
~
\begin{enumerate}
\item
When $M$ is a Fano manifold, i.e., a compact K\"{a}hler manifold
with positive first Chern class, Futaki and Morita showed that
(\cite{FM1}, \cite{FM2}), up to some constant, $f_{c_1^{n+1}}(A)$ is
nothing but the Futaki integral invariant with respect to the
holomorphic vector field $A$. This gives a geometric interpretation
of $f_{\varphi}(A)$ for $\varphi=c_1^{n+1}$ when $M$ is K\"{a}hler.
In contrast with this, it is somewhat surprising to see that
$f_{c_1c_n}(A)\equiv0$ for any compatible vector field $A$ on any
compact almost-complex manifold $M$.

\item
When $M$ is K\"{a}hler, $A$ is nondegenerate and $\textrm{zero}(A)$
only consists of isolated fixed points. Proposition \ref{prop} has
been obtained by Futaki and Morita (\cite{FM1}, \cite[p. 80]{Fu1}).
It is their this observation, together with their reinterpretation
of the Futaki integral invariant on Fano manifolds, that inspire our
this proposition and the current article.
\end{enumerate}
\end{remark}

As we have mentioned, Theorem \ref{localization1} says nothing for
those $\varphi$ whose degrees are larger than $n$. However, Futaki
and Morita noticed that (\cite{FM1}, \cite{FM2}, \cite[Chapter
5]{Fu}), for compact complex manifolds, Bott's original arguments in
\cite{Bo1} and \cite{Bo2} can also be applied to including
$\{\varphi~|~\textrm{deg}(\varphi)>n\}$. As an application, they
showed that, the famous Futaki integral invariant, which was
introduced by Futaki in \cite{Fu} and obstructs the existence of
K\"{a}hler-Einstein metrics on Fano manifolds, can be put into this
framework as a special case.

Now let us give a precise statement of Bott residue formula in this
strengthened version.

We denote by $I^{k}(gl(n,\mathbb{C}))$ $(0\leq k\leq n)$ the
$GL(n,\mathbb{C})$-invariant polynomial function of degree $k$ on
$gl(n,\mathbb{C})$, i.e., $\varphi\in I^{k}(gl(n,\mathbb{C}))$ means
that
$$\varphi~:~gl(n,\mathbb{C})\rightarrow\mathbb{C},$$
$$\varphi\big((a_{ij})_{n\times n}\big)=\sum\lambda_{i_1\cdots i_kj_1\cdots j_k}
a_{i_1j_1}\cdots a_{i_kj_k},$$

where
$$(a_{ij})_{n\times n}\in
gl(n,\mathbb{C})\qquad\text{and}\qquad\lambda_{i_1\cdots
i_kj_1\cdots j_k}\in\mathbb{C},$$

and satisfies
$$\varphi(PAP^{-1})=\varphi(A),\qquad \forall~ A\in
gl(n,\mathbb{C}),~ \forall ~P\in GL(n,\mathbb{C}).$$

It is well-known that
$$I^{\ast}(gl(n,\mathbb{C})):=\bigoplus_{k\geq
0}I^{k}(gl(n,\mathbb{C}))$$

is multiplicatively generated by $c_i$ $(0\leq i\leq n)$, which are
characterized by
$$\textrm{det}(I_n+tA)=:\sum_{i=0}^nt^i\cdot c_i(A),\qquad I_n=\textrm{$n\times n$ identity matrix}.$$

\begin{remark}
By a slight abuse of symbols, $c_i$ has at least four different
meanings in our article: the $i$-th elementary symmetric polynomial,
the $i$-th generator of $GL(n,\mathbb{C})$-invariant polynomial, the
$i$-th Chern form $c_i(\nabla):=c_i(\frac{\sqrt{-1}}{2\pi}R)$, and
the $i$-th Chern class of a complex vector bundle. The reason for
this abuse is clear to those who are familiar with the Chern-Weil
theory.
\end{remark}

Given $\varphi\in I^{\ast}(gl(n,\mathbb{C}))$, with the above
symbols and notation understood, we know that
$\varphi\big(L(X)+\frac{\sqrt{-1}}{2\pi}R\big)$ is a well-defined
differential form of possibly mixed degrees on $N$. The following
beautiful residue formula of Bott tells us that, under some
reasonable requirement on $X$, i.e., $X$ is \emph{non-degenerate}
(see Remark \ref{remarkafter} for a precise definition), the
evaluation of $\varphi\big(L(X)+\frac{\sqrt{-1}}{2\pi}R\big)$ on $N$
can be localized to that of $\textrm{zero}(X)$.

\begin{theorem}[Bott's residue formula, complex case]\label{localization2}
{\rm{With the above materials understood and assume that $X$ is
non-degenerate. For any $\varphi\in I^{\ast}(gl(n,\mathbb{C}))$, we
have
 \be\label{bottresidue}
\begin{split}
f_{\varphi}(X):&=\int_M\varphi\big(L(X)+\frac{\sqrt{-1}}{2\pi}R\big)\\
&=\sum_{Z\subset\textrm{zero}(X)}\int_Z
\frac{\varphi\big(L(X)\big|_Z+\frac{\sqrt{-1}}{2\pi}R\big|_Z\big)}
{\textrm{det}\big(L^{\mu}(X)+\frac{\sqrt{-1}}{2\pi}R^{\mu}\big)},\\
\end{split}\ee

where the sum is over all the connected components $Z$ of
$\textrm{zero}(X)$, $(\cdot)\big|_Z$ denotes the restriction to $Z$,
$L^{\mu}(X)\in\Gamma\big(\textrm{End}(\mu(Z))\big)$ is the induced
section of the normal bundle $\mu(Z)$ from $L(X)$, and $R^{\mu}$ is
the curvature form of $\mu(Z)$ with respect to the induced Hermitian
metric.}}
\end{theorem}

Some more remarks related to Theorem \ref{localization2} are in
order.

\begin{remark}\label{remarkafter}
~
\begin{enumerate}
\item
The precise meaning of \emph{non-degenerate} is that (\cite[p.
314]{Bo2}) each $Z$ is a complex submanifold of $N$ and the kernel
of the endomorphism $L(X)\big|_Z$ on $T^{1,0}N\big|_Z$ is precisely
$T^{1,0}Z$. Thus non-degeneracy guarantees that both the denominator
and the integral on the right hand side of (\ref{bottresidue}) be
well-defined. When this $X$ is Killing, which is the requirement in
Theorem \ref{localization1}, it is always non-degenerate according
to the \emph{compact} transformation group theory. So despite the
integrability condition, the assumption on the vector field in this
theorem is weaker than that in Theorem \ref{localization1}. The
reason is that the proof of the latter is based on the Lefschetz
fixed point formula of Aityah-Bott-Segal-Singer, which needs the
group acted on the manifold to be compact.\\

\item
When $\textrm{deg}(\varphi)\leq n$, we have
$$\int_M\varphi\big(L(X)+\frac{\sqrt{-1}}{2\pi}R\big)=\int_M\varphi(\frac{\sqrt{-1}}{2\pi}R)$$

and thus (\ref{bottresidue}) becomes
 \be\label{bottresidue2}
\int_M\varphi(\frac{\sqrt{-1}}{2\pi}R)
=\sum_{Z\subset\textrm{zero}(X)}\int_Z
\frac{\varphi\big(L(X)\big|_Z+\frac{\sqrt{-1}}{2\pi}R\big|_Z\big)}
{\textrm{det}\big(L^{\mu}(X)+\frac{\sqrt{-1}}{2\pi}R^{\mu}\big)}.
\nonumber\ee

This, via the Chern-Weil theory, exactly corresponds to the
localization formula in Theorem \ref{localization1}. This is what
Bott's original statement presents. Using the current version of
Bott's residue formula, Futaki-Morita showed that, up to some
constant factor, the original Futaki invariant is essentially equal
to $f_{c_1^{n+1}}(X)$ and so the calculation of the Futaki invariant
can also be localized to the zero point of the holomorphic vector
field explicitly by Theorem \ref{localization2}. Later, Bott's idea
was further extracted by Tian in \cite[\S~ 6]{Ti} to give a residue
formula of the Calabi-Futaki integral invariant, which obstrcuts the
existence of constant scalar curvature metrics in a given K\"{a}hler
class.\\

\item
A detailed proof of (\ref{bottresidue}) can be found in
\cite[Theorem 5.2.8]{Fu1}. However, an implicit but more concise
proof can also be found in \cite[p. 313]{Zhang}.
\end{enumerate}
\end{remark}

By virtue of the above discussion, combining Proposition \ref{prop}
with Theorem \ref{localization2} can immediately yield the following
result.
\begin{proposition}\label{prop2}
{\rm{Suppose $(N^{2n},~J,~g)$ is a compact Hermitian manifold, $X$ a
Killing  and real holomorphic vector field on $N$ and $\nabla$ the
Chern connection. Then we have
$$\int_Mc_1c_n
\big(L(X)+\frac{\sqrt{-1}}{2\pi}R\big)=\sum_Z\big[\textrm{tr}\big(L^{\mu}(X)\big)\cdot
e(Z)\big]=0.$$ }}
\end{proposition}

\section{Applications}\label{section3}
In this section, we will present two applications of Propositions
\ref{prop} and \ref{prop2}. The first one is the proof of Theorem
\ref{mainresult2}, which is almost immediate from Proposition
\ref{prop2}. The second one is based on an interesting observation
in \cite{FM1}.

First, we prove Theorem \ref{mainresult2}, which we restate here
again.
\begin{theorem}
{\rm{Suppose $(N^{2n},~J,~g)$ is a compact Hermitian manifold and
$X$ is a Killing and real holomorphic vector field. Then we have
$$\int_M\textrm{tr}\big(L(X)\big)\cdot c_n(\nabla)+
\int_Mc_1(\nabla)\cdot\textrm{det}
\big(L(X)+\frac{\sqrt{-1}}{2\pi}R\big)=0.$$

Here $c_i(\nabla):=c_i(\frac{\sqrt{-1}}{2\pi}R)$ is the $i$-th Chern
form with respect to the Chern connection $\nabla$.}}
\end{theorem}

\begin{proof}
By Proposition \ref{prop2} we have
$$\int_Mc_1c_n\big(L(X)+\frac{\sqrt{-1}}{2\pi}R\big)=0.$$

Note that
\be
\begin{split}&c_1c_n\big(L(X)+\frac{\sqrt{-1}}{2\pi}R\big)\\
=&\textrm{tr}\big(L(X)+\frac{\sqrt{-1}}{2\pi}R\big)\cdot
\textrm{det}\big(L(X)+\frac{\sqrt{-1}}{2\pi}R\big)\\
=&\big[\textrm{tr}\big(L(X)\big)+\textrm{tr}\big(\frac{\sqrt{-1}}{2\pi}R\big)\big]\\
&\cdot\big[\textrm{det}(\frac{\sqrt{-1}}{2\pi}R)+
\{\textrm{det}\big(L(X)+
\frac{\sqrt{-1}}{2\pi}R\big)\}^{(n-1)}+\textrm{lower degree
terms}\big]\\
=&\big[\textrm{tr}\big(L(X)\big)+c_1(\nabla)\big]\cdot
\big[c_n(\nabla)+ \{\textrm{det}\big(L(X)+
\frac{\sqrt{-1}}{2\pi}R\big)\}^{(n-1)}+\textrm{lower degree
terms}\big]\\
=&\textrm{tr}\big(L(X)\big)\cdot c_n(\nabla)+ c_1(\nabla)\cdot
\{\textrm{det}\big(L(X)+ \frac{\sqrt{-1}}{2\pi}R\big)\}^{(n-1)}.
\end{split}\nonumber\ee

Here $\{\cdots\}^{(n-1)}$ means the component of $(n-1,n-1)$-form in
$\{\cdots\}$.  Now the equality in the above theorem follows easily
from this deduction.
\end{proof}

\begin{remark}
It is clear that the statement itself is purely
differential-geometric. But the author does not know whether or not
this result could be proved by a direct differential-geometric
argument.
\end{remark}

Our second application in this section is based on an observation in
\cite{FM1}, which has been mentioned in Remark
\ref{remarkfutakimorita}.

Besides the reinterpretation of the Futaki invariant, Futaki and
Morita also proved many other interesting results related to
$f_{\varphi}(X)$. \cite{FM1} is an announcement of the properties of
these integral invariants, whose proofs are contained in \cite{FM2}
and chapter 5 of Futaki's book \cite{Fu1}. For example, they found
that $c_1c_{n}$ is the \emph{unique} monomial of degree $n+1$
satisfying $f_{c_1c_n}(X)\equiv0$ for \emph{any} non-degenerate
holomorphic vector field whose zero points are \emph{isolated} on
\emph{any K\"{a}hler} manifold. Now combining their observation and
our Proposition \ref{prop}, we have the following vanishing-type and
uniqueness result.

\begin{theorem}
{\rm{Among all the monomials of degree $n+1$
$$\{c_1^{s_1}\cdots
c_n^{s_n}~\big|~\sum_{i=0}^n i\cdot
s_i=n+1,~s_i\in\mathbb{Z},~s_i\geq 0\},$$

$c_1c_n$ is the \emph{unique} monomial which satisfies
$f_{c_1c_n}(A)=0$ for \emph{any} compatible vector field $A$ on
\emph{any} almost-complex manifold $(M^{2n},J)$.}}
\end{theorem}

\section{Proof of Theorem \ref{mainresult1}}\label{section4}
The proof of Theorem \ref{mainresult1} is an interesting application
of the rigidity property of the Hirzebruch $\chi_y$-genus. This
rigidity phenomenon, which was first observed by Lusztig in
\cite{Lu}, is a striking application of the general Lefschetz fixed
point formula developed by Atiyah, Bott, Segal and Singer. Recently
the author refines this observation and gives some applications to
symplectic geometry and related topics (\cite{Li}, \cite{Li3}).

The idea of the proof of Theorem \ref{mainresult1} basically follows
that of the main theorem in \cite{Li3}. However, for the reader's
convenience, we still sketch its proof.

As usual we use $\bar{\partial}$ to denote the $d$-bar operator
which acts on the complex vector spaces $\Omega^{p,q}(M)$ ($0\leq
p,q\leq n$) of $(p,q)$-type differential forms on $(M^{2n},J)$ in
the sense of $J$. The choice of the almost Hermitian metric $g$ on
$(M^{2n},J)$ enables us to define the Hodge star operator $\ast$ and
the formal adjoint $\bar{\partial}^{\ast}=-\ast\bar{\partial}~\ast$
of the $\bar{\partial}$-operator. Then for each $0\leq p\leq n$, we
have the following Dolbeault-type elliptic operator

\be\label{GDC}\bigoplus_{\textrm{$q$
even}}\Omega^{p,q}(M)\xrightarrow
{\bar{\partial}+\bar{\partial}^{\ast}} \bigoplus_{\textrm{$q$
odd}}\Omega^{p,q}(M),\ee

whose index is denoted by $\chi^{p}(M)$ in the notation of
Hirzebruch. We define the Hirzebruch $\chi_{y}$-genus,
$\chi_{y}(M)$, by

$$\chi_{y}(M):=\sum_{p=0}^{n}\chi^{p}(M)\cdot y^{p}.$$

The general form of the Hirzebruch-Riemann-Roch theorem allows us to
compute $\chi_y(M)$ in terms of the Chern numbers of $M$ as follows.

$$\chi_y(M)=\int_M\prod_{i=1}^n\frac{x_i(1+ye^{-x_i})}{1-e^{-x_i}}.$$

The proof of Theorem \ref{mainresult1} can be divided into the
following four steps, which we encode by four lemmas.\\

\emph{$\mathbf{STEP ~1.}$}

The first lemma we present below is refined from \cite{Lu}, which is
a beautiful application of the Lefschetz fixed point theorem for
elliptic complexes (\ref{GDC}).
\begin{lemma}
The following identity holds

\be\label{ABSS}\chi_y(M)\equiv\sum_Z\int_Z\big(\prod_{i=1}^r\alpha_i\frac{1+ye^{-\alpha_i}}{1-e^{-\alpha_i}}\big)\big(\prod
_{j=1}^{n-r}\frac{1+ye^{\sqrt{-1}\lambda_jt}e^{-\beta_j}}{1-e^{\sqrt{-1}\lambda_jt}e^{-\beta_j}}\big),\qquad
\forall ~t,\ee

i.e., the right-hand side of (\ref{ABSS}), when taken as a rational
function of $t$, is identically equal to $\chi_y(M)$.
\end{lemma}
\begin{proof}
The action of the one-parameter group $\textrm{exp}(tA)$ on $(M,J)$
can be lifted to the elliptic complex (\ref{GDC}). Thus we can
define the equivariant index $\chi^p(t,M)$ and the equivariant
$\chi_y$-genus $\chi_y(t,M):=\sum_{p=0}^n\chi^p(t,M)\cdot y^p$. The
Lefschetz fixed point formula of Atiyah-Bott-Segal-Singer (\cite[p.
562]{AS}) allows us to compute $\chi_y(t,M)$ in terms of the local
information around the fixed point set of the one-parameter group
action $\textrm{exp}(tA)$, which is exactly $\textrm{zero}(A)$, as
follows.

\be\label{ABSS2}\chi_y(t,M)=\sum_Z\int_Z\big(\prod_{i=1}^r\alpha_i\frac{1+ye^{-\alpha_i}}{1-e^{-\alpha_i}}\big)\big(\prod
_{j=1}^{n-r}\frac{1+ye^{\sqrt{-1}\lambda_jt}e^{-\beta_j}}{1-e^{\sqrt{-1}\lambda_jt}e^{-\beta_j}}\big).\ee

Note that the right-hand side of (\ref{ABSS2}) has well-defined
limits as $\sqrt{-1}t$ tends to $+\infty$ and $-\infty$:

$$\lim_{\sqrt{-1}t\rightarrow+\infty}\big(\textrm{RHS of (\ref{ABSS2})}\big)
=\sum_Z\chi_y(Z)(-y)^{d_{+}(Z)},$$
$$\lim_{\sqrt{-1}t\rightarrow
-\infty}\big(\textrm{RHS of (\ref{ABSS2})}\big)
=\sum_Z\chi_y(Z)(-y)^{d_{-}(Z)},$$

where $d_{+}(Z)$ \big(resp. $d_{-}(Z)$\big) is the number of
positive (resp. negative) numbers among the eigenvalues
$\lambda_1,\ldots,\lambda_{n-r}$. So the left-hand side of
(\ref{ABSS2}) and thus each $\chi^p(t,M)$ also have well-defined
limits as $\sqrt{-1}t$ tends to $+\infty$ and $-\infty$.

But by definition, for each $0\leq p\leq n$, $\chi^p(t,M)$ is the
trace of the action $\textrm{exp}(tA)$ on the complex representation
space
$\textrm{ker}(\bar{\partial}+\bar{\partial}^{\ast})-\textrm{coker}(\bar{\partial}+\bar{\partial}^{\ast})$.
Note that $A$ is Killing with respect to the metric $g$. Thus
$\chi^p(t,M)$ is of the following \emph{finite} sum

 $$\chi^p(t,M)=\sum_ia_i(p)\cdot\textrm{exp}\big(\sqrt{-1}t\theta_i(p)\big),
 \qquad a_i(p)\in\mathbb{Z}-\{0\},~\theta_i(p)\in\mathbb{R}.$$

So the only possibility that $\chi^p(t,M)$ has well-defined limits
as $\sqrt{-1}t$ tends to $+\infty$ and $-\infty$ is
$\theta_i(p)\equiv0$ and therefore $\chi^p(t,M)$ is a constant for
any $t$. This completes the proof of (\ref{ABSS}) and Step $1$.
\end{proof}

\emph{$\mathbf{STEP ~2.}$}

In this step we calculate a coefficient in the Taylor expansion of
(\ref{ABSS}) at $y=-1$. The motivation that inspires us to
investigate the coefficient of $y+1$ has been explained in details
in \cite[Section 3]{Li3}, which comes from another interesting
phenomenon of the Hirzebruch $\chi_y$-genus.

\begin{lemma}
Note that the right-hand side of (\ref{ABSS}) is a polynomial of
$y$. If we consider its Taylor expansion at $y=-1$, the coefficient
of the first order term $y+1$ is

\be\label{RHScoeff}\sum_Z\big[(\frac{r}{2}-n)e(Z)+e(Z)\sum_{j=1}^{n-r}
\frac{1}{1-e^{\sqrt{-1}\lambda_jt}}\big]\ee
\end{lemma}

\begin{proof}
From \cite[Lemma 2.3]{Li3} we know that

\be\label{xishu1}\prod_{i=1}^r\frac{\alpha_i(1+ye^{-\alpha_i})}{1-e^{-\alpha_i}}
=c_r(Z)+\big[c_{r-1}(Z)-\frac{r}{2}c_r(Z)\big]\cdot(y+1)+\cdots,\nonumber\ee
and

\be\label{xishu2}
\begin{split}\prod
_{j=1}^{n-r}\frac{1+ye^{\sqrt{-1}\lambda_jt}e^{-\beta_j}}
{1-e^{\sqrt{-1}\lambda_jt}e^{-\beta_j}}&=\prod_{j=1}^{n-r}
\big[1+\frac{e^{\sqrt{-1}\lambda_jt}e^{-\beta_j}}
{1-e^{\sqrt{-1}\lambda_jt}e^{-\beta_j}}(y+1)\big]\\
&=1+\big(\sum_{j=1}^{n-r}\frac{e^{\sqrt{-1}\lambda_jt}e^{-\beta_j}}
{1-e^{\sqrt{-1}\lambda_jt}e^{-\beta_j}}\big)(y+1)+\cdots\\
&=1+\big(\sum_{j=1}^{n-r}
\frac{e^{\sqrt{-1}\lambda_jt}}{1-e^{\sqrt{-1}\lambda_jt}}+\textrm{higher
degree terms}\big)\cdot(y+1)+\cdots. \nonumber\end{split}\ee

Combining these two expressions we can obtain that the coefficient
of $y+1$ on the right-hand side of (\ref{ABSS}) is

\be\begin{split}&\sum_Z\big[-\frac{r}{2}e(Z)+e(Z)\sum_{j=1}^{n-r}
\frac{e^{\sqrt{-1}\lambda_jt}}{1-e^{\sqrt{-1}\lambda_jt}}\big]\\
=&\sum_Z\big[-\frac{r}{2}e(Z)+e(Z)\sum_{j=1}^{n-r}
\big(-1+\frac{1}{1-e^{\sqrt{-1}\lambda_jt}}\big)\big]\\
=& \sum_Z\big[(\frac{r}{2}-n)e(Z)+e(Z)\sum_{j=1}^{n-r}
\frac{1}{1-e^{\sqrt{-1}\lambda_jt}}\big],\end{split}\nonumber\ee

which is precisely (\ref{RHScoeff}). This completes Step $2$.
\end{proof}

\emph{$\mathbf{STEP ~3.}$}

We carefully compare the coefficients of the term $y+1$ on both
sides of (\ref{ABSS}) and obtain the following result.

\begin{lemma}
 \be\label{comparecoeff}\begin{split}
&\sum_{\substack{Z\subset\textrm{zero}(A)\\e(Z)>0}}\big[e(Z)\sum_{j=1}^{n-r}\frac{1}{1-e^{\sqrt{-1}t\lambda_j}}\big]+
\sum_{\substack{Z\subset\textrm{zero}(A)\\e(Z)<0}}
\big[-e(Z)\sum_{j=1}^{n-r}\frac{1}
{1-e^{-\sqrt{-1}t\lambda_j}}\big]\\
\equiv& \frac{1}{2}\sum_{Z\subset
\textrm{zero}(A)}(n-r)|e(Z)|,\qquad \forall~t.\end{split}\ee

 Here
$``|\cdot|"$ means taking the absolute value.
\end{lemma}

\begin{proof}
The coefficient of the left-hand side of (\ref{ABSS}) is
$-\frac{n}{2}e(M)$ (\cite[Lemma 2.3]{Li3}), which, together with
(\ref{RHScoeff}), yields

\be\label{c1}-\frac{n}{2}e(M)\equiv\sum_Z\big[(\frac{r}{2}-n)e(Z)+e(Z)\sum_{j=1}^{n-r}
\frac{1}{1-e^{\sqrt{-1}\lambda_jt}}\big],\qquad\forall~ t.\ee

Note that we have the famous result $e(M)=\sum_Ze(Z)$, which, for
example, can be obtained by taking $y=-1$ in (\ref{ABSS}) as
$\chi_y(\cdot)\big|_{y=-1}=e(\cdot)$.

Using this identity $e(M)=\sum_Ze(Z)$ to substitute $e(M)$ in
(\ref{c1}), we obtain

\be\label{xishu4}\begin{split}&\frac{1}{2}\sum_Z(n-r)e(Z)\\
\equiv&\sum_Z\big[e(Z)\sum_{j=1}^{n-r}\frac{1}{1-e^{\sqrt{-1}\lambda_jt}}\big]\\
=&\sum_{\substack{Z\subset\textrm{zero}(A)\\e(Z)>0}}\big[e(Z)\sum_{j=1}^{n-r}\frac{1}{1-e^{\sqrt{-1}\lambda_jt}}\big]+
\sum_{\substack{Z\subset\textrm{zero}(A)\\e(Z)<0}}\big[-e(Z)\sum_{j=1}^{n-r}\frac{-1}{1-e^{\sqrt{-1}\lambda_jt}}\big]\\
=&\sum_{\substack{Z\subset\textrm{zero}(A)\\e(Z)>0}}\big[e(Z)\sum_{j=1}^{n-r}\frac{1}{1-e^{\sqrt{-1}\lambda_jt}}\big]+
\sum_{\substack{Z\subset\textrm{zero}(A)\\e(Z)<0}}\big[-e(Z)\sum_{j=1}^{n-r}(\frac{1}{1-e^{-\sqrt{-1}\lambda_jt}}-1)\big],~
\forall~t.
\end{split}\ee

Rewriting (\ref{xishu4}) slightly, we can easily yield
(\ref{comparecoeff}).
\end{proof}

\emph{$\mathbf{STEP ~4.}$}

 We can now complete the proof of Theorem
\ref{mainresult1}.
\begin{lemma}
Theorem \ref{mainresult1} holds.
\end{lemma}

\begin{proof}
Set $h:=\textrm{exp}(\sqrt{-1}t)$. Then (\ref{comparecoeff}) becomes

\be\label{comparecoeff1}\begin{split}&\sum_{\substack{Z\subset\textrm{zero}(A)\\e(Z)>0}}\big[e(Z)\sum_{j=1}^{n-r}\frac{1}{1-h^{\lambda_j}}\big]+
\sum_{\substack{Z\subset\textrm{zero}(A)\\e(Z)<0}}\big[-e(Z)\sum_{j=1}^{n-r}\frac{1}{1-h^{-\lambda_j}}\big]\\
\equiv& \frac{1}{2}\sum_{Z\subset\textrm{zero}(A)}(n-r)|e(Z)|,\qquad
\forall~h.\end{split}\ee

Now Theorem \ref{mainresult1} follows from (\ref{comparecoeff1}),
\cite[Lemma 2.4]{Li3} and \cite[Remark 2.5]{Li3}.
\end{proof}

\section*{Acknowledgments}

This work was initiated when the author was holding a JSPS
Postdoctoral Fellowship in the Departments of Mathematics at Tokyo
Metropolitan University (2011-2012) and Waseda University
(2012-2013) respectively. The author would like to thank the
Departments and his host Professor Martin A. Guest for their
hospitality during the author's stay in Japan.

\bibliographystyle{amsplain}

\begin{thebibliography}{10}

\bibitem{AS}
{M.F. Atiyah, I.M. Singer}:
\newblock {\em The index theory of elliptic operators:} III,
\newblock Ann. Math. {\bf 87} (1968), 546-604.

\bibitem{Bo1}
{R. Bott}:
\newblock {\em Vector fields and characteristic numbers},
\newblock Michigan Math. J. {\bf 14} (1967), 231-244.

\bibitem{Bo2}
{R. Bott}:
\newblock {\em A residue formula for holomorphic vector fields},
\newblock J. Differential Geom. {\bf 1} (1967), 311-330.

\bibitem{Fu}
{A. Futaki}:
\newblock {\em An obstruction to the existence of Einstein K\"{a}hler metrics},
\newblock Invent. Math. {\bf 73} (1983), 437-443.


\bibitem{Fu1}
{A. Futaki}:
\newblock {\em K\"{a}hler-Einstein Metrics and Integral Invariants},
\newblock Lecture Notes in Mathematics, {\bf 1314}. Springer-Verlag, Berlin,
1988.

\bibitem{FM1}
{A. Futaki, S. Morita}:
\newblock {\em Invariant polynomials on compact complex manifolds},
\newblock Proc. Japan Acad. {\bf 60} (1984), 369-372.

\bibitem{FM2}
{A. Futaki, S. Morita}:
\newblock {\em Invariant polynomials of the automorphism group of a compact complex manifold},
\newblock J. Differential Geom. {\bf 21} (1985), 135-142.

\bibitem{Ko}
{S. Kobayashi}:
\newblock {\em Transformation Groups in Differential Geometry},
\newblock Classics in Mathematics, Springer-Verlag, Berlin, 1995, Reprint of the 1972 edition.

\bibitem{Li}
{P. Li}:
\newblock {\em The rigidity of Dolbeault-type operators and symplectic circle actions},
\newblock Proc. Amer. Math. Soc. {\bf 140} (2012), 1987-1995.

\bibitem{Li2}
{P. Li}:
\newblock {\em Remarks on Bott residue formula and
Futaki-Morita integral invariants},
\newblock  Topology Appl. {\bf 160} (2013), 488-497.

\bibitem{Li3}
{P. Li}:
\newblock {\em An application of the rigidity of Dolbeault-type operators},
\newblock Math. Res. Letters. {\bf 20} (2013), 81-89.

\bibitem{Lu}
{G. Lusztig}:
\newblock {\em Remarks on the holomorphic Lefschetz numbers},
\newblock pp. 193-204 in: Analyse globale,(S\'{e}m. Math. Sup\'{e}rieures No.42,
1969). Presses Univ. Montr\'{e}al, Montreal, Que., 1971.

\bibitem{PT}
{A. Pelayo, S. Tolman}:
\newblock {\em Fixed points of symplectic periodic flows},
\newblock Ergodic Theory Dynam. Systems {\bf 31} (2011), 1237-1247

\bibitem{Ti}
{G. Tian}:
\newblock{\em K\"{a}hler-Einstein Metrics on Algebraic Manifolds},
\newblock Lecture Notes in Mathematics. 1646, Springer, Berlin-New York,
1996.

\bibitem{Zhang}
{W.P. Zhang}:
\newblock {\em A remark on a residue formula of Bott},
\newblock Acta Math. Sinica (N.S.) {\bf 6} (1990), 306-314.
\end{thebibliography}

\end{document}